\documentclass[11pt]{amsart}
\usepackage{amsthm,amsmath,stmaryrd,bbm,geometry,color}
\usepackage{amssymb}
\usepackage[english]{babel}
\usepackage[utf8]{inputenc}
\usepackage{graphicx}
\usepackage{verbatim}
\usepackage{enumitem}
\usepackage{mathtools}
\usepackage[titletoc]{appendix}
\usepackage{bm}
\usepackage{textcomp}
\usepackage[authoryear]{natbib}

\usepackage[dvipsnames]{xcolor}
\usepackage[hyperindex=true,frenchlinks=true,colorlinks=true,
citecolor=Mahogany,linkcolor=Red,urlcolor=Tan,linktocpage]{hyperref}

\usepackage{tikz}
\usetikzlibrary{shapes}
\usetikzlibrary{fit}
\usetikzlibrary{decorations.pathmorphing}

\title[Some notes on ergodic theorems  for
$U$-statistics of order $m$]{Some notes on ergodic theorem  for
$U$-statistics of order $m$ for stationary and not necessarily ergodic sequences}

\date{\today}

\address[$\dagger$]{Institut de Recherche Mathématique Avancée
UMR 7501, Université de Strasbourg and CNRS
7 rue René Descartes
67000 Strasbourg, France}
\author{Davide Giraudo}
\numberwithin{equation}{section}
\renewcommand{\leq}{\leqslant}
\renewcommand{\geq}{\geqslant}

\newtheorem{Theorem}{Theorem}[section]

\newtheorem{Lemma}[Theorem]{Lemma}

\newtheorem{Corollary}[Theorem]{Corollary}

\theoremstyle{remark}

\newcommand{\intent}[1]{\llbracket #1\rrbracket}

\newcommand{\Bca}{\mathcal{B}}

\newcommand{\Fca}{\mathcal{F}}

\newcommand{\Ica}{\mathcal{I}}

\newcommand{\Sca}{\mathcal{S}}

\newcommand \ens[1]{\left\{ #1\right\}}

\newcommand \R{\mathbb R}

\newcommand \N{\mathbb N}
\newcommand \PP{\mathbb P}
\newcommand{\el}{\mathbb L}
\newcommand{\E}[1]{\mathbb E\left[#1\right]}
\newcommand{\pr}[1]{\left(#1\right)}

\newcommand \Z{\mathbb Z}

\newcommand \abs[1]{\left|#1\right|}
\newcommand \eps{\varepsilon}

\newcommand{\norm}[1]{\left\lVert #1 \right\rVert}

\newcommand{\inc}{\operatorname{Inc}}
\newcommand{\ind}[1]{\mathbf{1}_{#1}}

\newcommand{\sgn}[1]{\operatorname{sgn}\pr{#1}}
\begin{document}

\begin{abstract}
 In this note, we give sufficient conditions for the almost sure 
and the convergence in $\el^p$ of a $U$-statistic of order $m$ built on a 
strictly stationary but not necessarily ergodic sequence.
\end{abstract}
\maketitle
\section{Introduction and main results}

\cite{MR26294} introduced the concept of $U$-statistics of order $m\in\N^*$, defined as follows:
if $\pr{X_i}_{i\geq 1}$ is strictly stationary sequence taking values in a measurable space $\pr{S,\Sca}$ and $h\colon S^m\to\R$, the $U$-statistic of kernel $h$ is given by 
\begin{equation}
U_{m,n,h}:=\frac 1{\binom mn}\sum_{\pr{i_\ell}_{\ell\in\intent{1,m}}\in\inc^m_n }
h\pr{X_{i_1},\dots,X_{i_m}},
\end{equation}
where $\intent{1,m}=\ens{k\in\N,1\leq k\leq m}$ and $\inc^m_n 
=\ens{\pr{i_\ell}_{\ell\in\intent{1,m}}, 1\leq i_1<i_2<\dots<i_m\leq n}$. 
If $\pr{X_i}_{i\geq 1}$ is i.i.d.\ and $\E{\abs{h\pr{X_1,\dots,X_m}}}$ is finite, 
then $U_{m,n,h}\to \E{h\pr{X_1,\dots,X_m}}$ a.s. and in $\el^1$. A natural question 
is whether for a strictly stationary sequence $\pr{X_i}_{i\geq 1}$, 
the sequence $\pr{U_{m,n,h}}_{n\geq m}$ converges almost surely or in $\el^1$ 
to some random variable. Assume first that $\pr{X_i}_{i\geq 1}$ is ergodic. 
It is shown in \cite{MR1363941} that if  $S=\R$, $\pr{X_i}_{i\geq 1}$ has common distribution $\PP_{X_0}$, $h$ is bounded and $\PP_{X_0}\times\dots\times \PP_{X_0}$ almost everywhere continuous, then 
\begin{equation}\label{eq:conv_ps_generale}
U_{m,n,h}\to \int h\pr{x_1,\dots,x_m}d\PP_{X_0}\pr{x_1}\dots 
d\PP_{X_0}\pr{x_m}\mbox{ a.s.}.
\end{equation}
Convergence in probability was also investigated in \cite{MR1979966}. A proof of 
\eqref{eq:conv_ps_generale} in the context of absolutely regular sequences has been given in \cite{MR1624866}. Moreover, Marcinkievicz law of large numbers for $U$-statistics of order two has been established in \cite{MR2571765} for absolutely regular sequences and \cite{MR4243516} for sequences expressable as functions of an independent sequence.

It is worth pointing out that in general, the sequence $\pr{U_{2,n,h}}_{n\geq 2}$ may fail to converge.  For instance, \cite{MR1363941}, Example 4.5, found a 
non-bounded 
kernel $h$ and a strictly stationary sequence such that $\limsup_{n\to\infty}U_{2,n,h}=\infty$.
 Moreover, the example given in Proposition~3 of 
\cite{dehling2023remarks} shows the existence of a bounded kernel $h$ 
and a stationary ergodic sequence $\pr{X_i}_{i\geq 1}$
such that a subsequence of $\pr{U_{2,n,h}}_{n\geq 2}$ converges to $0$ almost 
surely and an other subsequence of $\pr{U_{2,n,h}}_{n\geq 2}$ converges to $1$ 
almost surely. Also, as Proposition~4 shows, boundedness in $\el^1$ of 
$\pr{h\pr{X_1,X_j}}_{j\geq 2}$ plays a key role, otherwise, we can find a kernel $h$ and a strictly stationary sequence $\pr{X_i}_{i\geq 1}$ for which the sequence
 $\pr{U_{2,n,h}-\E{U_{2,n,h}}}_{n\geq 2}$ converges to a non-degenerate 
normal distribution.

Some results have been established in \cite{dehling2023remarks}, assuming that 
the strictly stationary sequence $\pr{X_i}_{i\geq 1}$ is ergodic.
\begin{enumerate}
\item If $S$ is a separable metric space, $h\colon S\times S\rightarrow \R$ is a symmetric kernel that is bounded and $\PP_{X_0}\times \PP_{X_0}$-almost everywhere continuous, then, as $n\rightarrow \infty$, $U_{2,n,h}\to \int 
h\pr{x,y}d\PP_{X_0}\pr{x}d\PP_{X_0}\pr{y}$ almost surely.
\item If $S=\R^d$, the family $\ens{h\pr{X_1,X_j},j\geq 1}$ is uniformly integrable, 
 $h$ is $\PP_{X_0}\times\dots\times \PP_{X_0}$ almost everywhere continuous and symmetric, then 
 \begin{equation}\label{eq:conv_thm_ergo_densite_bornee}
 \lim_{n\to\infty}\E{\abs{U_{2,n,h}-\int_{\R^d}\int_{\R^d}h\pr{x,y}d\PP_{X_0}\pr{x}d\PP_{X_0}\pr{y} 
   }}= 0.
\end{equation}
\item If $S=\R^d$, the family $\ens{h\pr{X_1,X_j},j\geq 1}$ is uniformly integrable, 
$\int_{\R^d}\int_{\R^d}\abs{h\pr{x,y}}d\PP_{X_0}\pr{x}d\PP_{X_0}\pr{y} $ is finite,
the random variable  $X_0$ has a bounded density 
with respect to the Lebesgue measure on $\R^d$ and  
 for each $k\geq 1$,  the vector $\pr{X_0,X_k}$ has a density $f_k$ with respect to the Lebesgue measure of $\R^d\times\R^d$ and $\sup_{k\geq 1}\sup_{s,t\in\R^d}f_k\pr{s,t}$ is finite, then \eqref{eq:conv_thm_ergo_densite_bornee} holds.
\end{enumerate}

Such results lead us to consider the following extensions. The case of $U$-statistics of order two has been adressed and we may want to extend these results to $U$-statistics of arbitrary order, especially because such mathematical object is widely used in statistics, for instance in \cite{MR3127883} for the distance covariance and
 stochastic geometry (see \cite{MR3585402}).  Moreover, it is a natural question to 
see what happens in the non-ergodic case. It is natural to consider a 
decomposition of $\Omega$ into ergodic components and use the results of 
the ergodic case to each of them. However, the 
multiple integral expression of the limit does not give a simple expression. 
Moreover, the assumptions of the ergodic case 
for each ergodic component, namely, almost everywhere continuity (for the product law of the marginal distribution) of the kernel and and assumption on density of the vector $\pr{X_{i_1},\dots,X_{i_m}}$
does not seem to give a tractable condition. Instead,  we will use 
the following approach: when $h$ is symmetric and bounded, the convergence of 
the considered $U$-statistic is viewed as the convergence 
of random product measures toward a product of random measures (deterministic 
measures in the ergodic case). When we make an assumption on the density of  
$\pr{X_{i_1},\dots,X_{i_m}}$, we approximate 
$h$ by linear combinations of products of indicator functions. This approach has similarities with the one used in \cite{MR3256808}. The case 
of products of indicators follows then from an application of the 
usual ergodic theorem.

We will assume that the strictly stationary sequence is such that  $X_i=X_0\circ 
T^i$, where $T\colon\Omega\to\Omega$ is a measure preserving map. 
By \cite{MR0229267}, page 107, we know that we can find a random 
variable $X$ and a measure preserving map $T$ on $\R^{\Z}$ such that the 
sequences $\pr{X_i}_{i\in\Z}$ and $\pr{X\circ T^i}_{i\in\Z}$ have the same 
distribution. We will study almost sure convergence and in $\el^p$ 
sequences of the form $\pr{g_n\pr{X_1,\dots,X_n}}_{n\geq 1}$, where $g_n\colon 
S^n\to\R$, which has the same distribution as 
$\pr{g_n\pr{X\circ T^1,\dots,X\circ T^n}}_{n\geq 1}$. Therefore, their almost 
sure convergence  are equivalent, as well as their convergence in $\el^p$. To 
ease the notations, we will write $X=X_0$.

We will study the almost sure convergence of $\pr{U_{m,n,h}}_{n\geq m}$ 
and the convergence in $\el^p$. We will denote by $\norm{Z}_p:=
\pr{\E{\abs{Z^p}}}^{1/p}$ the norm of a real-valued random variable $Z$. 

It turns out that the limit will be expressed as an integral with respect to products of
a random measure defined as follows:
\begin{equation}\label{eq:def_mu_omega}
\mu_\omega\pr{A}=\E{\ind{X_0\in A}\mid\Ica}\pr{\omega}, A\in\Bca\pr{S},
\end{equation}
where $\Ica$ denotes the $\sigma$-algebra of invariant sets, that is, the sets $E$ such that $T^{-1}E=E$.
The limit of $U$-statistics will be expressed as integral with respect to 
product measure of $\mu_\omega$, which lead us to define 
\begin{equation}\label{eq:definition_of_I_h_omega}
I_m\pr{S,h,\omega}:=\int_{S^m}h\pr{x_1,\dots,x_m}d\mu_\omega\pr{x_1}
\dots d\mu_\omega\pr{x_m}.
\end{equation}
We will also define as $I_m\pr{S,h,\cdot}$ the random variable 
given by
\begin{equation}\label{eq:definition_of_I_h_omega_dot}
I_m\pr{S,h,\cdot} \colon \omega\mapsto   I_m\pr{S,h,\omega}.
\end{equation}  

Some assumption will be made on the set of discontinuity points of $h$, which will be denoted by $D\pr{h}$. 

\subsection{Almost sure convergence}

Our first result deals with the almost sure convergence of a $U$-statistic 
under the assumption of boundedness of the kernel and negligibility of the set of discontinuity with respect to the product of the marginal law.

\begin{Theorem}\label{thm:ergodic_theorem_sym_ps}
Let $\pr{S,d}$ be a separable metric space, let 
$\pr{X_i}_{i\in\Z}=\pr{X_0\circ T^i}_{i\in\Z}$ be a strictly stationary 
sequence. Suppose that $h\colon S^m\to\R$ satisfies the following assumptions:
\begin{enumerate}[label=(A.1.\arabic*)]
\item\label{ass:h_sym} $h$ is symmetric, that is, $h\pr{x_{\sigma\pr{1}},\dots,x_{\sigma\pr{m}}}=h\pr{x_1,\dots,x_m}$ for each $x_1,\dots,x_m\in S$ and each bijective 
$\sigma\colon\intent{1,m}\to\intent{1,m}$,
\item\label{ass:h_borne} $h$ is  bounded and
\item\label{ass:h_cont} for almost every $\omega\in\Omega$,
$I_m\pr{S,\ind{D\pr{h}},\omega}=0$, where $D\pr{h}$ denotes
the set of discontinuity points of $h$.
\end{enumerate}
Then for almost every $\omega\in\Omega$, 
the following convergence holds:
\begin{equation}\label{eq:conv_ps}
\lim_{n\to\infty} U_{m,n,h}\pr{\omega}
=I_m\pr{S,h,\omega} ,
\end{equation}
where $I_m\pr{S,h,\omega} $ is defined as in \eqref{eq:definition_of_I_h_omega}.
\end{Theorem}
This result extends Theorem~1 in \cite{dehling2023remarks} in two directions: first, the case of $U$-statistics of arbitrary order are considered. Second, we address here the not necessarity ergodic case.

When $\pr{X_{i}}_{i\geq 1}$ is ergodic, the measure $\mu_\omega$ is simply 
the distribution of $X_0$ hence the right hand side of \eqref{eq:conv_ps} 
can be simply expressed as $\E{h\pr{X_1^{\pr{1}},\dots,X_1^{\pr{m}}}}$, where 
$X_1^{\pr{1}},\dots,X_1^{\pr{m}}$ are independent copies of $X_1$. 

The symmetry assumption is needed in order to relate $U_{m,n,h}$ to a sum over a  rectangle and then see the convergence in \eqref{eq:conv_ps} as a convergence in distribution of product of random measures.

\subsection{Convergence in $\el^p$, $p\geq 1$}
In this subsection, we present sufficient conditions for the convergence 
in $\el^p$ of $\pr{U_{m,n,h}}_{n\geq 1}$.

We start by mentioning the following consequence of Theorem~\ref{thm:ergodic_theorem_sym_ps}.
\begin{Corollary}\label{cor:conv_Lp_sym}
Let$\pr{S,d}$ be a separable metric space, let $\pr{X_i}_{i\in\Z}=\pr{X_0\circ 
T^i}_{i\in\Z}$ be a strictly stationary sequence and let $p\geq 1$. Suppose that 
$h\colon S^m\to\R$ satisfies the following assumptions:
\begin{enumerate}[label=(A.2.\arabic*)]
\item\label{ass:h_sym_cor} $h$ is symmetric, that is, $h\pr{x_{\sigma\pr{1}},\dots,x_{\sigma\pr{m}}}=h\pr{x_1,\dots,x_m}$ for each $x_1,\dots,x_m\in S$ and each bijective 
$\sigma\colon\intent{1,m}\to\intent{1,m}$,
\item\label{ass:h_cor_UI} the family $\ens{\abs{h\pr{X_{i_1},\dots, 
X_{i_m}}}^p, 1\leq i_1<\dots<i_m}$ is uniformly integrable.
\item \label{ass:h_cor_limit_Lp} the following integral is finite:
\begin{equation}
\int_\Omega\int_{S^m} \abs{h\pr{x_1,\dots,x_m}}^p
d\mu_{\omega}\pr{x_1} \dots d\mu_{\omega}\pr{x_m}d\PP\pr{\omega}.
\end{equation}
\item\label{ass:h_cont_cor} for almost every $\omega\in\Omega$,
$I_m\pr{S,\ind{D\pr{h}},\omega}=0$, where $D\pr{h}$ denotes
the set of discontinuity points of $h$.
\end{enumerate}
Then the following convergence takes place:
\begin{equation}
\lim_{n\to\infty} \norm{U_{m,n,h} -I_m\pr{S,h,\cdot}}_p=0,
\end{equation}
where $I_m\pr{S,h,\cdot}$ is defined as in \eqref{eq:definition_of_I_h_omega_dot}.
\end{Corollary}

One can wonder what happens if we remove the symmetry assumption. 
\begin{Theorem}\label{thm:conv_Lp_ergodique}
Let $\pr{S,d}$ be a separable metric space, let $\pr{X_i}_{i\in\Z}=\pr{X_0\circ 
T^i}_{i\in\Z}$ be a strictly stationary sequence taking values in $S$ and let 
$p\geq 1$. Suppose that $h\colon S^2\to\R$
and $\pr{X_0\circ T^i}_{i\in\Z}$ satisfy the following assumptions:
\begin{enumerate}[label=(A.3.\arabic*)]
\item\label{ass:UI_ergodique} the collection 
$\ens{\abs{h\pr{X_{i},X_{j}}}^p,1\leq i<j}$ 
is uniformly integrable.
\item\label{ass:h_cont_Lp} for almost every $\omega\in\Omega$, 
$I_m\pr{S,\ind{D\pr{h}},\omega}=0$, where $D\pr{h}$ denotes 
the set of discontinuity points of $h$. 
\item\label{ass:h_copies_ordre_p}  the following integral is finite: 
\begin{equation}
 \int_\Omega \int_{S^2}\abs{h\pr{x,y}}^pd\mu_\omega\pr{x}d\mu_\omega\pr{y}d\PP\pr{\omega}. 
\end{equation}
\end{enumerate}
Then the following convergence takes place:
\begin{equation}\label{eq:conv_Lp_non_sym}
\lim_{n\to\infty}\norm{U_{2,n,h}-I_2\pr{S,h,\cdot}}  _p=0,
\end{equation}
where  $I_2\pr{S,h,\cdot}$ is defined as in \eqref{eq:definition_of_I_h_omega_dot}.
\end{Theorem}
This improves Theorem~2 in \cite{dehling2023remarks} under assumption (A.1) in the paper, since we do not require symmetry of the kernel.

One may wonder why we do not present a similar result for $U$-statistics of 
order $m$. A first idea would be an argument by induction on the dimension. In 
order to perform the induction step, say from $m=2$ to $m=3$, we would need to 
show, after a use of the weighted ergodic theorem,  
the convergence in  $\el^p$ of $\binom{n}{2}^{-1}\sum_{1\leq i<j\leq n}h\pr{X_{-j},X_{-i},X_0}$. Since we assume uniform integrability, it 
suffices to show the almost sure convergence, which could be established 
by seeing this almost sure convergence as that of a product of random measures. But without symmetry, we do not know whether the almost sure convergence of the sequence of random measures 
$\binom{n}{2}^{-1}\sum_{1\leq i<j\leq n}\delta_{\pr{X_{-j},X_{-i}}}$ takes place.

Let us now state a result on the convergence in $\el^p$ without imposing any continuity of the kernel, but making assumptions on the distribution 
of the vectors $\pr{X_{i_1},\dots,X_{i_m}}$ .

\begin{Theorem}\label{thm:conv_Lp_non-ergodique_densite}
Let $\pr{X_i}_{i\in\Z}=\pr{X_0\circ T^i}_{i\in\Z}$ be a strictly stationary 
sequence taking values in $\R^d$ and let $p\geq 1$. Suppose that $h\colon 
\pr{\R^d}^m\to\R$
and $\pr{X_0\circ T^i}_{i\in\Z}$ satisfy the following assumptions:
\begin{enumerate}[label=(A.4.\arabic*)]
\item\label{ass:Ui_dens} the collection 
$\ens{\abs{h\pr{X_{i_1},\dots,X_{i_m}}}^p,1\leq i_1<\dots<i_m}$ 
is uniformly integrable.
\item\label{ass:densite_bornee} for each $\pr{i_\ell}_{\ell\in\intent{1,m}}$ such that $1\leq i_1<\dots<i_m$, the vector 
$\pr{X_{i_1},\dots,X_{i_m}}$ has a density $f_{i_1,\dots,i_m}$  
and there exists a $q_0>1$ such that 
\begin{equation}
M_1:=\sup_{\pr{i_\ell}_{\ell\in\intent{1,m}}:1\leq i_1<\dots<i_m}
\int_{\pr{\R^d}^m}f_{i_1,\dots,i_m}\pr{t_1,\dots,t_m}^{q_0}dt_1\dots dt_m<\infty.
\end{equation}
\item\label{ass:den_bornee_m_omega} For almost every $\omega$, the measure $\mu_{\omega}$ defined as in \eqref{eq:def_mu_omega} admits a density 
$f_\omega$ with respect to the Lebesgue measure and there exists 
a set $\Omega'$ having probability one and $q_1>1$ for which 
\begin{equation}
M_2:=\sup_{\omega\in\Omega'}\int_{\R^d}f_{\omega}\pr{t}^{q_1}dt<\infty.
\end{equation}
\item \label{ass:inte_limite} the following integral is finite:
\begin{equation}
\int_{\Omega}\int_{\pr{\R^d}^m}\abs{h\pr{x_1,\dots,x_m}}^pd\mu_\omega\pr{x_1}\dots d\mu_{\omega}\pr{x_m}d\PP\pr{\omega} .
\end{equation}
\end{enumerate}
Then the following convergence hold:
\begin{equation}\label{eq:thm_erg_densite}
\lim_{n\to\infty}\norm{U_{m,n,h}-I_m\pr{\R^d,h,\cdot} }_p=0,
\end{equation}
where $I_m\pr{\R^d,h,\cdot}$ is defined as in \eqref{eq:definition_of_I_h_omega_dot}.
\end{Theorem}

Assumption~\ref{ass:densite_bornee} is needed in order 
to approximate $h$ by a linear combination of indicator functions 
of produts of Borel sets, uniformly with respect to the distribution of 
$\pr{X_{i_1},\dots,X_{i_m}}$. 

Our Theorem~\ref{thm:conv_Lp_non-ergodique_densite} improves Theorem~2 in \cite{dehling2023remarks} under assumption (A.2) in the following directions. First, we provide a result for $U$-statistics of arbitrary order. Second, the not necessarily ergodic case is addressed. Third, even in the ergodic case, our assumption only require a uniform control on the 
$\el^{q_1}$ norm of the densities instead of a uniform bound.

\subsection{Examples}\label{subsec:examples}
In this Subsection, we give Examples where the conditions of 
Corollary~\ref{cor:conv_Lp_sym}, Theorems~\ref{thm:ergodic_theorem_sym_ps} and 
\ref{thm:conv_Lp_non-ergodique_densite} are satisfied.
\begin{Corollary}\label{cor:kernel_test_symmetry}
 Let $h\colon\R^3\to\R$ be the kernel defined as 
 \begin{equation}\label{eq:kernel_test_symmetry}
  h\pr{x_1,x_2,x_3}=\sgn{2x_1-x_2-x_3}+\sgn{2x_2-x_1-x_3}+\sgn{2x_3-x_1-x_2},
 \end{equation}
where $\sgn{x}$ equals $1$ if $x>0$, $-1$ if $x<0$ and $\sgn{0}=0$. 
Let $\pr{X_i}_{i\in\Z}=\pr{X_0\circ T^i}_{i\in\Z}$ be a strictly stationary 
sequence of real valued random variables such that for each $x_0\in\R$, 
$\PP\pr{X_0=x_0}=0$. Then for almost every $\omega\in\Omega$, the 
convergence 
\begin{equation}
 U_{3,n,h}\pr{\omega}\to I_3\pr{\R,h,\omega}
\end{equation}
takes place, where $I_3\pr{\R,h,\omega}$ is defined as in 
\eqref{eq:definition_of_I_h_omega}.
\end{Corollary}
The kernel $h$ defined by \eqref{eq:kernel_test_symmetry} is used in order to
test symmetry of a distribution.
\\
The next example deals with the distance covariance, which is used in order
to test independence between two samples.
\begin{Corollary}\label{cor:distance_covariance}
 Let $\pr{S,d}$ be a separable metric space and let $p\geq 1$.
 Define $f\colon S^4\to \R$ by 
 \begin{equation}
  f\pr{z_1,z_2,z_3,z_4}=d\pr{z_1,z_2}-d\pr{z_1,z_3}-d\pr{z_2,z_4}+
  d\pr{z_3,z_4},
 \end{equation}
 the kernel $g\colon \pr{S^2}^6\to\R$ by 
\begin{equation}
 g\pr{\pr{x_1,y_1},\dots, \pr{x_6,y_6}}=f\pr{x_1,x_2,x_3,x_4}
 f\pr{y_1,y_2,y_3,y_4}
\end{equation}
and its symmetrized version 
\begin{equation}
 h\pr{\pr{x_1,y_1},\dots, \pr{x_6,y_6}}=\frac 1{6!}\sum_{\sigma\in\Sca_6}
  g\pr{\pr{x_{\sigma\pr{1}},y_{\sigma\pr{1}}},\dots, 
\pr{x_{\sigma\pr{6}},y_{\sigma\pr{6}}}}.
\end{equation}
Let $\pr{\pr{X_i,X'_i}}_{i\in\Z}=\pr{\pr{X_0,X'_0}\circ T^i}_{i\in\Z}$ be a 
strictly stationary sequence with values in $S^2$, where $\pr{X_0\circ 
T^i}_{i\in\Z}$ is independent of $\pr{X'_0\circ 
T^i}_{i\in\Z}$.
 Suppose that 
 $\E{d\pr{X_0,x_0}^p}+\E{d\pr{X'_0,x_0}^p}$ is finite for some (hence all)
$x_0\in S$.
Then the following convergence takes place 
\begin{equation}
 \lim_{n\to\infty}\norm{U_{6,n,h}-I_6\pr{S^2,h,\cdot}}_p=0,
\end{equation}
where $I_6\pr{S^2,h,\cdot}$ is defined by 
\begin{equation}
 I_6\pr{S^2,h,\omega}=\int_{S^6}h\pr{\pr{x_1,y_1},\dots,\pr{x_6,y_6}}
 d\mu_\omega\pr{x_1,y_1}\dots d\mu_\omega\pr{x_6,y_6}
\end{equation}
and for $A\in\Bca\pr{S^2}$, 
\begin{equation}
 \mu_\omega\pr{A}=\E{\ind{\pr{X_0,X'_0}\in A } \mid\Ica }\pr{\omega}.
\end{equation}
\end{Corollary}
Finally, let us give a framework where
Theorem~\ref{thm:conv_Lp_non-ergodique_densite} applies.
\begin{Corollary}\label{cor:gaussian}
 Let $h\colon\R^m\to \R$ be a measurable map and let $p\geq 1$.
Let $\pr{X_i}_{i\in\Z}$ be a strictly stationary Gaussian sequence. 
Denote by $\Sigma\pr{i_1,\dots,i_m}$ the covariance matrix of the vector 
$\pr{X_{i_1},\dots,X_{i_m}}$. Suppose that 
\begin{equation}\label{eq:det_matrice_cov}
 \inf_{1\leq i_1<\dots<i_m}\det\pr{\Sigma\pr{i_1,\dots,i_m}}>0,
\end{equation}
\begin{equation}\label{eq:cond_ergodicite}
 \lim_{N\to\infty}\frac 1N\sum_{i=1}^N\abs{\operatorname{Cov}\pr{X_0,X_i}}
 =0 \mbox{ and}
\end{equation}
\begin{equation}\label{eq:integrability_example_gaussian}
 \E{\abs{h\pr{X_0^{\pr{1} },\dots,X_0^{\pr{m} }  }}^p  }< \infty,
\end{equation}
where $X_0^{\pr{1} },\dots, X_0^{\pr{m} }$ are independent copies of $X_0$
and that \ref{ass:Ui_dens} holds. Then 
\eqref{eq:thm_erg_densite} takes place.
\end{Corollary}

\section{Proofs}

\subsection{Proof of Theorem~\ref{thm:ergodic_theorem_sym_ps}}
The symmetry assumption guarantees the following decomposition
\begin{equation}
U_{m,n,h}=\frac{1}{m!\binom{n}m}\sum_{i_1,\dots,i_m=1}^n 
h\pr{X_{i_1},\dots,X_{i_m}} - \frac 
1{ m! \binom{n}m}\sum_{\pr{i_\ell}_{\ell\in\intent{1,m}}\in J_n  
}h\pr{X_{i_1},\dots,X_{i_m}}, 
\end{equation}
where $J_n$ denotes the set of elements $\pr{i_\ell}_{\ell\in\intent{1,m}}
\in\intent{1,n}^m$ for which there exist at least two distinct indexes $\ell$ and $\ell'$ for which $i_\ell=i_{\ell'}$. Since $h$ is bounded and 
$\operatorname{Card}\pr{J_n}/\binom{n}m$ goes to $0$ as $n$ goes to infinity, 
it suffices to prove that for almost every $\omega\in\Omega$, 
\begin{equation}
\lim_{n\to\infty}\frac 1{n^m}\sum_{i_1,\dots,i_m=1}^n 
h\pr{X_{i_1}\pr{\omega},\dots,X_{i_m}\pr{\omega}}=
\int_{S^m}h\pr{x_1,\dots,x_m}d\mu_\omega\pr{x_1}\dots d\mu_{\omega}\pr{x_m},
\end{equation}
where $\mu_\omega$ is defined as in \eqref{eq:def_mu_omega}.
Observe that for each $\omega\in\Omega$, 
\begin{equation}
\frac 1{n^m}\sum_{i_1,\dots,i_m=1}^n 
h\pr{X_{i_1}\pr{\omega},\dots,X_{i_m}\pr{\omega}}=
\int_{S^m}h\pr{x_1,\dots,x_m}d\nu_{n,\omega}\pr{x_1} \dots d\nu_{n,\omega}\pr{x_m},
\end{equation}
where 
\begin{equation}
\nu_{n,\omega}=\frac 1n\sum_{i=1}^n \delta_{X_i\pr{\omega}}.
\end{equation}

Separability of $S$ guarantees the existence of a countable collection 
$\pr{f_k}_{k\geq 1}$ of continuous and bounded functions from $S$ to $\R$ 
such that a sequence $\pr{\mu_n}_{n\geq 1}$ of probability measures converges weakly to 
a probability measure $\mu$ if and only if for each $k\geq 1$, 
$\int f_kd\mu_n\to \int f_kd\mu$. By the ergodic theorem, we know that for each 
$k\geq 1$, there exists a set $\Omega_k$ having probability one for which 
the convergence
\begin{equation}
\lim_{n\to\infty}\int f_k\pr{x}d \nu _{n,\omega}
=\lim_{n\to\infty}\frac 1n\sum_{j=1}^n f_k\pr{X_j\pr{\omega}}
=\E{f_k\pr{X_0}\mid\Ica}\pr{\omega}=\int f_k\pr{x}d\mu_{\omega}\pr{x}.
\end{equation}
holds for each $\omega\in\Omega_k$. Therefore, for each $\omega$ belonging to the set of probability one 
$\Omega':=\bigcap_{k\geq 1}\Omega_k$, the sequence 
$\pr{ \nu_{n,\omega}}_{n\geq 1}$ 
converges weakly to $\mu_\omega$.

Recall that Theorem 3.2 (page 21) of \cite{MR0233396} shows that 
if $\mu_n\to\mu$ and $\mu'_n\to\mu'$ in distribution on metric spaces $S_1$ and $S_2$ respectively, then 
$\mu_n\times\mu'_n\to \mu\times \mu'$ in distribution on $S_1\times S_2$. Applying 
inductively this result and using assumptions~\ref{ass:h_borne} and 
\ref{ass:h_cont} shows that for each $\omega\in\Omega'$, \eqref{eq:conv_ps} 
holds. Indeed, we know that if $S'$ is a separable metric space, 
$\pr{m_n}_{n\geq 1}$ is a sequence of probability measures  which converges 
weakly to $m$ and $g\colon S'\to\R$ is bounded and $m\pr{D\pr{g}}=0$, then 
$\int g dm_n\to \int g dm$. We use this for each fixed $\omega\in\Omega'$ with 
$M_n=\nu_{n,\omega}\times\dots\times\nu_{n,\omega}$, $S'=S^m$ and $g=h$.

\subsection{Proof of Corollary~\ref{cor:conv_Lp_sym}}

Let $h_R$ be as in \eqref{eq:def_de_hR}. Observe that 
assumption \ref{ass:h_cor_limit_Lp} guarantee that $I_m\pr{S,h,\cdot}$ defined as 
in \eqref{eq:definition_of_I_h_omega} belongs to $\el^p$. Moreover, 
the triangle inequality implies 
\begin{equation*}
\norm{U_{m,n,h}-I_m\pr{S,h,\cdot}}_p\leq \norm{U_{m,n,h}- U_{m,n,h_R}}_p
+\norm{U_{m,n,h_R}-I_m\pr{S,h_R,\cdot}}_p+\norm{I_m\pr{S,h,\cdot}-I_m\pr{S,h_R,\cdot}}_p.
\end{equation*}
Using assumption~\ref{ass:h_cor_UI} combined with the triangle inequality, 
one gets 
\begin{equation*}
\lim_{R\to\infty}\sup_{n\geq m}\norm{U_{m,n,h}- U_{m,n,h_R}}_p=0.
\end{equation*}
 Moreover, assumption \ref{ass:h_cor_limit_Lp} combined with monotone convergence shows that $$\lim_{R\to\infty}\norm{I_m\pr{S,h,\cdot}-I_m\pr{S,h_R,\cdot}}_p=0$$
 hence it suffices to show that for each fixed $R>0$, 
 $\norm{U_{m,n,h_R} -I_m\pr{S,h_R,\cdot}}_p\to 0$ as $n$ goes to infinity. 
This follows 
 from an application of Theorem~\ref{thm:ergodic_theorem_sym_ps} 
 with $h$ replaced by $h_R$ (note that continuity of $ \phi_R$ guarantees 
 that $D\pr{h_R}\subset D\pr{h}$), which gives that $U_{m,n,h_R}\to 
 I_m\pr{S,h_R,\cdot}$ 
 almost surely and the dominated convergence theorem allows to conclude.

\subsection{Convergence of weighted averages}

The proof of Theorems~\ref{thm:conv_Lp_ergodique} and~\ref{thm:conv_Lp_non-ergodique_densite} rests on 
weighted versions of the ergodic theorem, which read as follows.

\begin{Lemma}\label{lem:weighted_averatges}
Let $T$ be a measure preserving map on the probability space 
$\pr{\Omega,\Fca,\PP}$ and let $\Ica$ be the $\sigma$-algebra of $T$ invariance sets. 
Let $p\geq 1$ and let $\pr{f_j}_{j\geq 1}$ be a sequence of functions
such that $\norm{f_j-f}_p\to 0$. Then for each $m\geq 0$, the following convergence holds:
\begin{equation}\label{eq:weighted_ergodic_theorem_ordre_m}
\lim_{n\to\infty}\norm{\frac 1{\binom n{m+1}}\sum_{j=m+1}^n\binom{j-1}m f_j\circ T^j-\E{f\mid\Ica}}_p=0.
\end{equation}
\end{Lemma}
 
\begin{proof}
First observe that since $T$ is measure preserving, 
\begin{multline}
\norm{\frac 1{\binom n{m+1}}\sum_{j=m+1}^n\binom{j-1}m f_j\circ T^j-\E{f\mid\Ica}}_p\\
\leq \frac 1{\binom n{m+1}}\sum_{j=m+1}^n\binom{j-1}m
\norm{f_j-f}_p+
\norm{\frac 1{\binom n{m+1}}\sum_{j=m+1}^n\binom{j-1}m f\circ T^j-\E{f\mid\Ica}}_p.
\end{multline}
The first term goes to zero as $n$ goes to infinity from the elementary fact 
that $\sum_{i=1}^nc_i x_i/\pr{\sum_{j=1}^n c_j} \to 0$ if $c_j>0$ and 
$\sum_{j=1}^n c_j\to \infty$. For the second term, we assume for 
without loss of generality that $\E{f\mid\Ica}=0$, otherwise, we replace $f$ by $f-\E{f\mid\Ica}$. Let $S_j:=\sum_{i=1}^j f\circ T^i$. Then 
\begin{equation}
\sum_{j=m+1}^n\binom{j-1}m f\circ T^j
=\sum_{j=m+1}^n\binom{j-1}m S_j-\sum_{j=m}^{n-1}\binom{j}mS_j
\end{equation}
and it follows that 
\begin{equation}\label{eq:step_proof_ergodic_weighted}
\norm{\frac 1{\binom n{m+1}}\sum_{j=m+1}^n\binom{j-1}m f\circ T^j }_p\leq \frac{\binom{n-1}m }{\binom n{m+1}}\norm{S_n}_p+
 \frac 1{\binom n{m+1}}
 \sum_{j=m}^{n-1}\pr{\binom{j}m-\binom{j-1}m}\norm{S_j}_p.
\end{equation}
Since $\norm{S_n}_p/n\to 0$, the first term of the right hand side of \eqref{eq:step_proof_ergodic_weighted} goes to $0$ as $n$ goes to infinity. For the 
second term, one has for each $m\leq R\leq n-1$ that 
\begin{multline}
 \frac 1{\binom n{m+1}}
 \sum_{j=m}^{n-1}\pr{\binom{j}m-\binom{j-1}m}\norm{S_j}_p
 \leq  \frac 1{\binom n{m+1}}
 \sum_{j=m}^{R}\pr{\binom{j}m-\binom{j-1}m}\norm{S_j}_p\\
 + \frac 1{\binom n{m+1}}
 \sum_{j=R}^{n-1}j\pr{\binom{j}m-\binom{j-1}m}\sup_{k\geq R}\frac{\norm{S_k}_p}k
\end{multline}
hence 
\begin{equation}
\limsup_{n\to\infty} \frac 1{\binom n{m+1}}\sum_{j=m}^{n-1}\pr{\binom{j}m-\binom{j-1}m}\norm{S_j}_p
 \leq\sup_{k\geq R}\frac{\norm{S_k}_p}k
\end{equation}
and we conclude using again that $\frac{\norm{S_k}_p}k\to 0$ as $k$ goes to infinity.
\end{proof}

\subsection{Proof of Theorem~\ref{thm:conv_Lp_ergodique}}
The proofs will lead us to consider truncated versions 
of the kernel $h$. Define for each fixed $R>0$ the maps $\phi_R\colon\R\to\R$ by
\begin{equation}
\phi_R\pr{t}:=
\begin{cases}
-R&\mbox{ if }t<-R,\\
t&\mbox{ if }-R\leq t<R,\\
 R&\mbox{ if }t\geq R 
  \end{cases}
\end{equation}
and $h_R\colon S^m\to\R$ by 
\begin{equation}\label{eq:def_de_hR}
h_R\pr{x_1,\dots,x_m}:=\phi_R\pr{h\pr{x_1,\dots,x_m}},\quad x_1,\dots,x_m\in S.
\end{equation}
Then $\abs{h_R}$ is bounded by $R$ and  since $D\pr{h_R}\subset D\pr{h}$, the 
equality
$\PP_{X_0}\times 
\PP_{X_0}\pr{D\pr{h_R}}=0$ holds.  We claim that it suffices to prove that 
\eqref{eq:conv_Lp_non_sym} holds for each $R$ with $h$ replaced by $h_R$. Indeed, by the triangle inequality, 
\begin{equation}
\sup_{n\geq 2}\norm{U_{2,n,h}-U_{2,n,h_R}}_p\leq 
\sup_{1\leq i<j}\norm{h\pr{X_{i},X_{j}}
\ind{\abs{h\pr{X_i,X_{j}}}>R}}_p
\end{equation}
and
\begin{multline}
\int_\Omega \abs{
\int_{S^2}h\pr{x,y}d\mu_{\omega}\pr{x}
d\mu_{\omega}\pr{y}-
\int_{S^2}h_R\pr{x,y}d\mu_{\omega}\pr{x}
d\mu_{\omega}\pr{y}
}^pd\PP\pr{\omega}
\\
\leq 
\int_\Omega
\int_{S^2} \abs{h\pr{x,y}}^p \ind{\abs{h\pr{x,y}}>R}  d\mu_{\omega}\pr{x}
d\mu_{\omega}\pr{y}d\PP\pr{\omega},
\end{multline}
hence assumptions~\ref{ass:UI_ergodique} and \ref{ass:h_copies_ordre_p} 
allows us to choose $R$ making the previous quantities as small as we wish.
Defining
\begin{equation}
d_{j,R}:=\frac 1j\sum_{i=1}^{j-1}h_R\pr{X_{-i},X_0},
\end{equation}  
we get that 
\begin{equation}\label{eq:U_2n_en_terms_de_djR}
U_{2,n,h_R}=\frac 1{\binom n2}\sum_{j=2}^n jd_{j,R}\circ T^j.
\end{equation}
We first show that there exists a set of probability one $\Omega'$ such that for each $\omega\in \Omega'$, 
\begin{equation}\label{eq:conv_djR}
d_{j,R}\pr{\omega}\to \int_{S^2}h_R\pr{x,y}d\mu_{\omega}\pr{x}d\delta_{X_0\pr{\omega}}=:Y_R\pr{\omega}.
\end{equation}
First, separability of $S$ guarantees the existence of a countable collection 
$\pr{f_k}_{k\geq 1}$ of continuous and bounded functions from $S$ to $\R$ 
such that a sequence $\pr{\mu_n}_{n\geq 1}$ of probability measures converges weakly to 
a probability measure $\mu$ if and only if for each $k\geq 1$, 
$\int f_kd\mu_n\to \inf f_kd\mu$.

Taking $\mu_{n,\omega}:= n^{-1}\sum_{i=1}^n \delta_{X_{-i}\pr{\omega}}$, the ergodic theorem furnishes for each $k\geq 1$ a set $\Omega_k$ having probability one for which 
the convergence
\begin{equation}
\lim_{n\to\infty}\int f_k\pr{x}d\mu_{n,\omega}
=\lim_{n\to\infty}\frac 1n\sum_{j=1}^n f_k\pr{X_j\pr{\omega}}
=\E{f_k\pr{X_0}\mid\Ica}\pr{\omega}=\int f_k\pr{x}d\mu_{\omega}\pr{x}.
\end{equation}
holds for each $\omega\in\Omega_k$. Consequently, for each $\omega\in\Omega':=\bigcap_{k\geq 1}\Omega_k$, one has $\mu_{n,\omega}\to 
\mu_{\omega}$ weakly in $S$ and by Theorem 3.2 (page 21) of \cite{MR0233396}, 
we get that $\mu_{n,\omega}\times\delta_{X_0\pr{\omega}}\to 
\mu_{\omega}\times\delta_{X_0\pr{\omega}}$ weakly in $S^2$. Since $h_R$ is 
bounded and for almost every $\omega\in \Omega$, 
$\mu_{\omega}\times\delta_{X_0\pr{\omega}}
\pr{D\pr{h_R}}=0$, we get \eqref{eq:conv_djR}. 
 Moreover, $\abs{d_{j,R}}\leq R$ hence by dominated convergence, 
 \begin{equation}\label{eq:convergence_de_djR}
 \lim_{j\to\infty}\norm{d_{j,R} - \int_{S^2}h_R\pr{x,y}d\mu_{\omega}\pr{x}d\delta_{X_0\pr{\omega}}}_p=0.
 \end{equation}
By \eqref{eq:U_2n_en_terms_de_djR} and the fact that $T$ 
is measure preserving, we infer that 
\begin{equation}\label{eq:conv_Ustat_ordre_2_etape}
\norm{U_{2,n,h_R} -\E{Y_R\mid\Ica}}_p
\leq \frac 1{\binom n2}\sum_{j=2}^nj\norm{ d_{j,R}-Y_R }_p
+\norm{\frac 1{\binom n2}\sum_{j=2}^nj Y_R\circ T^j -\E{Y_R\mid\Ica}}_p.
\end{equation}
An application of \eqref{eq:conv_djR} combined with the dominated convergence theorem 
shows that the first term of the right hand side of 
\eqref{eq:conv_Ustat_ordre_2_etape} goes to $0$ as $n$ goes to infinity. 
Then Lemma~\ref{lem:weighted_averatges} with $m=1$   shows that $\norm{U_{2,n,h_R} -
\E{Y_R\mid\Ica}}_p\to 0$. It remains to check that 
\begin{equation}
\E{Y_R\mid\Ica}\pr{\omega}=\int_{S^2}h\pr{x,y}d\mu_{\omega}\pr{x}
d\mu_{\omega}\pr{y}.
\end{equation}
Observe that by the ergodic theorem, $\E{Y_R\mid\Ica}\pr{\omega}
=\lim_{n\to\infty}  n^{-1}\sum_{k=1}^n Y_R\pr{T^k\omega}$. 
Since $\mu_{T^k\omega}=\mu_{\omega}$, it follows that 
\begin{equation}
\E{Y_R\mid\Ica}\pr{\omega}=\lim_{n\to\infty}\frac 1n
\sum_{k=1}^n\int_{S^2}h_R\pr{x,y}d\mu_{\omega}\pr{x}d\delta_{X_k\pr{\omega}}\pr{y}.
\end{equation}
Using similar arguments as before gives that for almost every $\omega$, 
$\mu_\omega\times \pr{n^{-1}\sum_{k=1}^n\delta_{X_k\pr{\omega}}}$ 
converges in distribution to $\mu_\omega\times\mu_\omega$. 
This ends the proof of Theorem~\ref{thm:conv_Lp_ergodique}.

\subsection{Proof of Theorem~\ref{thm:conv_Lp_non-ergodique_densite}}

We start by proving Theorem~\ref{thm:conv_Lp_non-ergodique_densite} 
in the case where $h\pr{x_1,\dots,x_m}=\prod_{\ell=1}^m
\ind{x_\ell\in A_\ell}$. We show by induction over $m$ that 
if $A_\ell,\ell\in\intent{1,m}$ are Borel subsets of $\R^d$ and 
$\pr{X\circ T^i}_{i\in\Z}$ a stationary sequence with invariance 
$\sigma$-algebra $\Ica$, then 
\begin{equation}\label{eq:thm_ergo_indic}
\lim_{n\to\infty}\norm{\frac 1{\binom nm}\sum_{\pr{i_\ell}_{\ell\in\intent{1,m}}\in\inc^m_n}
\prod_{\ell=1}^m
\ind{X_{i_\ell}\in A_\ell} -\prod_{\ell=1}^m\E{\ind{X_0\in A_\ell} 
\mid\Ica}  }_p=0.
\end{equation}
The case $m=1$ is a direct consequence of the ergodic theorem. 
Let us show the case $m=2$. We start from 
\begin{equation}
 \frac{1}{\binom n2}\sum_{1\leq i<j\leq n}\ind{X_i\in A_1}\ind{X_j\in A_2}  
\\
 = \frac{1}{\binom n2}\sum_{j=2}^n j \pr{
\frac 1j\sum_{ k =1}^{j-1}\ind{X_{- k }\in A_1}\ind{X_0\in A_2}  }\circ 
T^j,
\end{equation}
where the change of index $k=j-i$ for a fixed $j$ has been done.
Let $f_j:=\frac 1j\sum_{k=1}^{j-1}\ind{X_{-k}\in A_1}\ind{X_0\in 
A_2}$. 
By the ergodic theorem, $f_j\to \E{\ind{X_0\in A_1}\mid\Ica}\ind{X_0\in A_2}$, which gives \eqref{eq:thm_ergo_indic} for $m=2$. 

Suppose now that \eqref{eq:thm_ergo_indic} holds for each Borel subset $A_1,\dots,A_m$ of $\R^d$ and each strictly stationary sequence $\pr{X_0\circ T^i}_{i\in\Z}$. Let $A_1,\dots,A_{m+1}$ be Borel subsets 
of $\R^d$ and let $\pr{X_0\circ T^i}_{i\in\Z}$ be a strictly stationary sequence. We start from 
\begin{multline}
\frac 1{\binom n{m+1}}\sum_{\pr{i_\ell}_{\ell\in\intent{1,m+1}}\in\inc^{m+1}_n}
\prod_{\ell\in\intent{1,m+1}}
\ind{X_{i_\ell}\in A_\ell}\\
=\frac 1{\binom n{m+1}}\sum_{j=m+1}^n 
\pr{\ind{X_0\in A_{m+1}}\sum_{\pr{i_\ell}_{\ell\in\intent{1,m}}\in\inc^m_{j-1}} \prod_{\ell\in\intent{1,m}   }\ind{X_{i_\ell-j}\in A_\ell   }}\circ T^j.
\end{multline}

Define
\begin{equation}
f_j:=\frac 1{\binom{j-1}m}\ind{X_0\in A_{m+1}}\sum_{\pr{i_\ell}_{\ell\in\intent{1,m}}\in\inc^m_{j-1}} \prod_{\ell\in\intent{1,m}   }\ind{X_{i_\ell-j}\in A_\ell   }.
\end{equation}
Doing the changes of index $k_1=j-i_m,\dots,k_m=j-i_1$, the previous expression can be rewritten as 
\begin{equation}
f_j=\ind{X_0\in A_{m+1}}\frac 1{\binom{j-1}m}\sum_{\pr{k_\ell}_{\ell\in\intent{1,m}}\in\inc^m_{j-1}} \prod_{\ell\in\intent{1,m}   }\ind{X_{-k_\ell}\in A_{m-\ell+1}   }
\end{equation} 
and using the induction assumption, we derive that 
\begin{equation}
\lim_{j\to\infty}\norm{f_j-  \ind{X_0\in A_{m+1}}\prod_{\ell\in \intent{1,m}}
\E{\ind{X_0\in A_{m-\ell+1}}\mid \Ica}}_p.
\end{equation}
Then we conclude by \eqref{eq:weighted_ergodic_theorem_ordre_m}.

We now show \eqref{eq:thm_erg_densite} in the general case. 
Fix a positive $\eps$ and define for a positive $K$
\begin{equation}
h^{\pr{K}}\pr{x_1,\dots,x_m}=h\pr{x_1,\dots,x_m}
\ind{\abs{h\pr{x_1,\dots,x_m}}\leq K}
\prod_{\ell=1}^m\ind{\abs{x_\ell}_d\leq K},
\end{equation}
where $\abs{\cdot}_d$ denotes the Euclidean norm on $\R^d$.  
Observe that by the triangle inequality, 
\begin{multline}
\norm{U_{m,n,h}-U_{m,n,h^{\pr{K}}}}_p
\leq \sup_{1\leq i_1< \dots<i_m}
\norm{h\pr{X_{i_1},\dots,X_{i_m}}\ind{\abs{h\pr{X_{i_1},\dots,X_{i_m}}}>K} }_p\\
+\sum_{\ell=1}^m \sup_{1\leq i_1< \dots<i_m}
\norm{h\pr{X_{i_1},\dots,X_{i_m}}\ind{\abs{ x_\ell}_d>K} }_p,
\end{multline}
hence by assumption~\ref{ass:Ui_dens}, we can find $K'$ such that for each $K\geq K'$, 
\begin{equation}
\sup_{n\geq m}\norm{U_{m,n,h}-U_{m,n,h^{\pr{K}}}}_p\leq \eps.
\end{equation}
Moreover, by assumption~\ref{ass:inte_limite}, we can choose $K''$ such that for each $K\geq K''$, 
\begin{equation}
\int_{\Omega}
I_m\pr{\R^d, \abs{h-h^{\pr{K}}}^p,\omega}d\PP\pr{\omega}
\leq\eps^p .
\end{equation}
 Let $K_0=\max\ens{K',K''}$. Observe that in  assumptions~\ref{ass:densite_bornee} and \ref{ass:den_bornee_m_omega} , 
 we can assume without loss of generality that $q_0=q_1$. By standard results in measure theory, we know that we can find an integer $J$, 
 constants $c_1,\dots,c_J$ and Borel subsets $A_{\ell,j}$, $\ell\in\intent{1,m}, j\in\intent{1,J}$ such that 
 \begin{equation}\label{eq:approx_h_K_comb_ind}
 \int_{\pr{\R^d}^m}\abs{h^{\pr{K_0}}\pr{x_1,\dots,x_m}-
\widetilde{ h}^{\pr{K_0}}\pr{x_1,\dots,x_m}}^{p\frac{q_0}{q_0-1}}
dx_1\dots dx_m 
<\pr{M_1+M_2}^p\eps^{p\frac{q_0}{q_0-1}} ,
 \end{equation}
where 
\begin{equation}\label{eq:expression_h_K_comb_ind}
\widetilde{ h}^{\pr{K_0}}\pr{x_1,\dots,x_m}
=\sum_{j=1}^Jc_j\prod_{\ell=1}^m
\ind{x_\ell \in A_{\ell,j}}.
\end{equation}
Notice that for each $1\leq i_1<\dots<i_m$,
\begin{multline}
\norm{h^{\pr{K_0}}\pr{X_{i_1},\dots,X_{i_m}}-
\widetilde{ h}^{\pr{K_0}}\pr{X_{i_1},\dots,X_{i_m}}}_p^p\\
= \int_{\pr{\R^d}^m}\abs{h^{\pr{K_0}}\pr{x_1,\dots,x_m}-
\widetilde{ h}^{\pr{K_0}}\pr{x_1,\dots,x_m}}^{p }
f_{i_1,\dots,i_m}\pr{x_1,\dots,x_m}
dx_1\dots dx_m
\end{multline}
hence using Hölder's inequality, \eqref{eq:approx_h_K_comb_ind} 
and assumption~\ref{ass:densite_bornee}, we derive that 
\begin{equation}
\sup_{1\leq i_1<\dots<i_m}\norm{h^{\pr{K_0}}\pr{X_{i_1},\dots,X_{i_m}}-
\widetilde{ h}^{\pr{K_0}}\pr{X_{i_1},\dots,X_{i_m}}}_p \leq \eps
\end{equation}
and by the triangle inequality, 
\begin{equation}
\sup_{N\geq m}\norm{U_{m,N,h^{\pr{K_0}}}-U_{m,N,\widetilde{h}^{\pr{K_0}}}}_p\leq\eps.
\end{equation}
Moreover, using Hölder's inequality, we find that 
\begin{equation}
\int_{\Omega}
I_m\pr{\R^d,\abs{h^{\pr{K_0}}-\widetilde{h}^{\pr{K_0}}}^p,\omega}d\PP\pr{\omega}
\leq \eps^p.
\end{equation}
As a consequence, 
\begin{align*}
\norm{U_{m,n,h}-I_m\pr{\R^d,h,\cdot}}_p
&\leq \sup_{N\geq m}\norm{U_{m,N,h}-U_{m,N,h^{\pr{K_0}}}}_p
+ \sup_{N\geq m}\norm{U_{m,N,h^{\pr{K_0}}}-U_{m,N,\widetilde{h}^{\pr{K_0}}}}_p
\\
&+\norm{U_{m,n,\widetilde{h}^{\pr{K_0}}}-I_m\pr{\R^d,\widetilde{h}^{\pr{K_0}},\cdot}}_p
+\norm{I_m\pr{\R^d,\widetilde{h}^{\pr{K_0}},\cdot}-
I_m\pr{\R^d,h^{\pr{K_0}},\cdot}}_p\\
&+\norm{I_m\pr{\R^d,\ h ^{\pr{K_0}},\cdot}-
I_m\pr{\R^d,h ,\cdot}}_p.
\end{align*}
By \eqref{eq:thm_ergo_indic} and \eqref{eq:expression_h_K_comb_ind}, 
we can find $n_0$ such that for each $n\geq n_0$, 
$\norm{U_{m,n,\widetilde{h}^{\pr{K_0}}}-I_m\pr{\R^d,\widetilde{h}^{\pr{K_0}
},\cdot}}_p\leq \eps$ hence 
we derive that for such $n$'s, $\norm{U_{m,n,h}-I_m\pr{\R^d,h,\cdot}}_p\leq 4\eps$. This ends the proof of Theorem~\ref{thm:conv_Lp_non-ergodique_densite}.

\subsection{Proof of the results of Subsection~\ref{subsec:examples}}
  \begin{proof}[Proof of Corollary~\ref{cor:kernel_test_symmetry}]
    This is an application of Theorem~\ref{thm:ergodic_theorem_sym_ps}.
    Assumption~\ref{ass:h_sym}  and \ref{ass:h_borne} are clear. In order to
    check Assumption~\ref{ass:h_cont}, we notice that
    \begin{multline}
     D\pr{h}\subset \ens{\pr{x_1,x_2,x_3}\in \R^3,2x_1=x_2+x_3} \cup
     \ens{\pr{x_1,x_2,x_3}\in \R^3,2x_2=x_1+x_3}\\
     \cup\ens{\pr{x_1,x_2,x_3}\in \R^3,2x_3=x_1+x_2}.
    \end{multline}
    It suffices to show that for almost every $\omega$, $\int_{\R^3}
\ind{2x_1=x_2+x_3}d\mu_\omega\pr{x_1}
    d\mu_\omega\pr{x_2}d\mu_\omega\pr{x_3}=0$, since the treatment of the other
terms is completely similar. Observe that
$\int \ind{2x_1=x_2+x_3}d\mu_\omega\pr{x_3}=
\E{\ind{X_0=2x_1-x_2}\mid\Ica }\pr{\omega}$, and the expectation of
$\E{\ind{X_0=2x_1-x_2}\mid\Ica }=0$ is zero hence this random variable equals
$0$  almost surely.
  \end{proof}
  \begin{proof}[Proof of Corollary~\ref{cor:distance_covariance}]
    This is an application of Corollary~\ref{cor:conv_Lp_sym}.
Assumption~\ref{ass:h_sym_cor} is by definition satisfied. Let us check
assumption~\ref{ass:h_cor_UI}. By definition of $h$, it suffices to check that
the family $\ens{\abs{f\pr{X_{i_1},X_{i_2},X_{i_3},X_{i_4}}
    f\pr{X'_{i_1},X'_{i_2},X'_{i_5},X'_{i_6}}}^p,1\leq i_1<\dots<i_6}$
is uniformly integrable. Using the elementary fact that if $\pr{Y_j}_{j\geq 1}$
is independent of $\pr{Y'_j}_{j\geq 1}$ and both sequences are uniformly
integrable, then so is $\pr{Y_jY'_j}_{j\geq 1}$, it suffices to prove that
$\ens{\abs{f\pr{X_{i_1},X_{i_2},X_{i_3},X_{i_4}}}^p,1\leq i_1<\dots<i_4}$ is
uniformly integrable, since the argument for the other term is completely 
similar. By
definition of $f$ and the triangle inequality,
$\abs{f\pr{X_{i_1},X_{i_2},X_{i_3},X_{i_4}}}\leq 2
\pr{d\pr{X_{i_1},x_0 }+d\pr{X_{i_2},x_0 }+d\pr{X_{i_3},x_0 } +d\pr{X_{i_4},x_0
}}$ and uniform integrability follows from finiteness of
$\E{d\pr{X_0,x_0}^p}$. Let us check \ref{ass:h_cor_limit_Lp}. Using the 
definition of $h$ and $f$ and the triangle inequality, it suffices to prove 
that $\int_{S^2\times 
S^2}d\pr{x_1,x_0}^pd\pr{y_2,x_0}^pd\mu_\omega\pr{x_1,y_1}d\mu_\omega\pr{x_2,y_2}
$ and $\int_{S^2}d\pr{x_1,x_0}^pd\pr{y_1,x_0}^pd\mu_\omega\pr{x_1,y_1}$  are 
finite. This follows from the fact that $\int 
u\pr{x}v\pr{y}d\mu_{\omega}\pr{x,y}=\E{u\pr{X_0}v\pr{X'_0}\mid\Ica 
}\pr{\omega}$ and integrability of $d\pr{X_0,x}^pd\pr{X'_0,x}^p$. Continuity 
of $h$ guarantees \ref{ass:h_cont_cor}.
  \end{proof}
  \begin{proof}[Proof of Corollary~\ref{cor:gaussian}]
  Condition \eqref{eq:det_matrice_cov} shows that the density of the vector 
  $\pr{X_{i_1},\dots,X_{i_m}}$ is bounded over $\R^m$ and the bound is 
uniform with respect to $\pr{i_1,\dots,i_m}$. Therefore, 
Assumption~\ref{ass:densite_bornee} is satisfied. 
  By \eqref{eq:cond_ergodicite}, the sequence $\pr{X_0\circ T^i}_{i\in\Z}$ is 
ergodic (see \cite{MR0285046}). \ref{ass:den_bornee_m_omega} holds because 
$f_\omega$ is the density of $X_0$ and \ref{ass:inte_limite} follows from 
\eqref{eq:integrability_example_gaussian}. 
  \end{proof}
\textbf{Acknowledgement}    The author would like to thank the referees 
for many valuable comments that improved the note.

\def\cprime{$'$}


\begin{thebibliography}{9}
\providecommand{\natexlab}[1]{#1}
\providecommand{\url}[1]{\texttt{#1}}
\expandafter\ifx\csname urlstyle\endcsname\relax
  \providecommand{\doi}[1]{doi: #1}\else
  \providecommand{\doi}{doi: \begingroup \urlstyle{rm}\Url}\fi

\bibitem[Aaronson et~al.(1996)Aaronson, Burton, Dehling, Gilat, Hill, and
  Weiss]{MR1363941}
J.~Aaronson, R.~Burton, H.~Dehling, D.~Gilat, T.~Hill, and B.~Weiss.
\newblock Strong laws for {$L$}- and {$U$}-statistics.
\newblock \emph{Trans. Amer. Math. Soc.}, 348\penalty0 (7):\penalty0
  2845--2866, 1996.
\newblock ISSN 0002-9947.
\newblock URL \url{https://doi.org/10.1090/S0002-9947-96-01681-9}.

\bibitem[Arcones(1998)]{MR1624866}
M.~A. Arcones, \emph{The law of large numbers for {$U$}-statistics under
  absolute regularity}, Electron. Comm. Probab. \textbf{3} (1998), 13--19.
  \MR{1624866}


\bibitem[Billingsley(1968)]{MR0233396}
P.~Billingsley.
\newblock \emph{{Convergence of probability measures}}.
\newblock John Wiley \& Sons Inc., New York, 1968.

\bibitem[Borovkova et~al.(2002)Borovkova, Burton, and Dehling]{MR1979966}
S.~Borovkova, R.~Burton, and H.~Dehling.
\newblock From dimension estimation to asymptotics of dependent
  {$U$}-statistics.
\newblock In \emph{Limit theorems in probability and statistics, {V}ol. {I}
  ({B}alatonlelle, 1999)}, pages 201--234. J\'{a}nos Bolyai Math. Soc.,
  Budapest, 2002.

\bibitem[Breiman(1968)]{MR0229267}
    L.~Breiman
    \newblock \emph{Probability},
    \newblock Addison-Wesley Publishing Co., Reading, Mass.-London-Don
              Mills, Ont., 1968. 

\bibitem[Dehling et~al.(2023)]{dehling2023remarks}
H. Dehling, D. Giraudo, and D. Voln\'{y}.
\newblock Some remarks on the ergodic theorem for {$U$}-statistics, 2023.
\newblock C. R. Math. Acad. Sci. Paris 361 (2023), 1511–1519
\newblock URL \url{https:// 10.5802/crmath.494}

\bibitem[Dehling and Sharipov(2009)]{MR2571765}
H.~G. Dehling and O.~Sh. Sharipov, \emph{Marcinkiewicz-{Z}ygmund
  strong laws for {$U$}-statistics of weakly dependent observations}, Statist.
  Probab. Lett. \textbf{79} (2009), no.~19, 2028--2036. \MR{2571765}


\bibitem[Denker and Gordin(2014)]{MR3256808}
M. Denker and M. Gordin.
\newblock Limit theorems for von {M}ises statistics of a measure preserving
  transformation.
\newblock \emph{Probab. Theory Related Fields}, 160\penalty0 (1-2):\penalty0
  1--45, 2014.
\newblock ISSN 0178-8051.
\newblock URL \url{https://doi.org/10.1007/s00440-013-0522-z}.

\bibitem[Giraudo(2021)]{MR4243516}
D. Giraudo, \emph{Limit theorems for {$U$}-statistics of {B}ernoulli data},
  ALEA Lat. Am. J. Probab. Math. Stat. \textbf{18} (2021), no.~1, 793--828.
  \MR{4243516}


\bibitem[Hoeffding(1948)]{MR26294}
W. Hoeffding.
\newblock A class of statistics with asymptotically normal distribution.
\newblock \emph{Ann. Math. Statistics}, 19:\penalty0 293--325, 1948.
\newblock ISSN 0003-4851.
\newblock URL \url{https://doi.org/10.1214/aoms/1177730196}.

\bibitem[Lachi\`eze-Rey and Reitzner(2016)]{MR3585402}
R. Lachi\`eze-Rey and M. Reitzner.
\newblock {$U$}-statistics in stochastic geometry.
\newblock In \emph{Stochastic analysis for {P}oisson point processes}, volume~7
  of \emph{Bocconi Springer Ser.}, pages 229--253. Bocconi Univ. Press, 2016.

\bibitem[Lyons(2013)]{MR3127883}
R. Lyons.
\newblock Distance covariance in metric spaces.
\newblock \emph{Ann. Probab.}, 41\penalty0 (5):\penalty0 3284--3305, 2013.
\newblock ISSN 0091-1798.
\newblock URL \url{https://doi.org/10.1214/12-AOP803}.

\bibitem[Maruyama(1970)]{MR0285046}
  G. Maruyama
  \newblock Infinitely divisible processes.
  \newblock \emph{Teor. Verojatnost. i Primenen.}, 15, 3--23 (1970) 
  \newblock URL \url{https:doi/10.1137/1115001}

\end{thebibliography}
\end{document}